\documentclass[12pt]{article}
\usepackage{amsmath}
\usepackage{url}
\newcommand{\ang}[1]{\langle#1\rangle}

\newcommand{\xvec}[1]{\ifcase 3{#1} {\ang {x_1,x_2,x_3} } \else 
\ifcase 4{#1} {\ang{x_1,x_2,x_3,x_4}} \else {\ang {x_1,\ldots,x_{#1}}}\fi\fi}
\newcommand{\yvec}[1]{\ifcase 3{#1} {\ang {y_1,y_2,y_3} } \else 
\ifcase 4{#1} {\ang{y_1,y_2,y_3,y_4}} \else {\ang {y_1,\ldots,y_{#1}}}\fi\fi}
\newcommand{\zvec}[1]{\ifcase 3{#1} {\ang {z_1,z_2,z_3} } \else 
\ifcase 4{#1} {\ang{z_1,z_2,z_3,z_4}} \else {\ang {z_1,\ldots,z_{#1}}}\fi\fi}
\newcommand{\vecc}[2]{\ifcase 3{#2} {\ang { {#1}_1,{#1}_2,{#1}_3 } } \else
\ifcase 4{#1} {\ang { {#1}_1,{#1}_2,{#1}_3,{#1}_{4} } }
\else {\ang { {#1}_1,\ldots,{#1}_{#2}}}\fi\fi}
\newcommand{\veccd}[3]{\ifcase 3{#2} {\ang { {#1}_{{#3}1},{#1}_{{#3}2},{#1}_{{#3}3} } } \else
\ifcase 4{#1} {\ang { {#1}_{{#3}1},{#1}_{{#3}2},{#1}_{#3}3},{#1}_{{#3}4} }
\else {\ang { {#1}_{{#3}1},\ldots,{#1}_{{#3}{#2}}}}\fi\fi}
\newcommand{\veccz}[2]{\ifcase 3{#2} {\ang { {#1}_0,{#1}_2,{#1}_3 } } \else
\ifcase 4{#1} {\ang { {#1}_0,{#1}_2,{#1}_3,{#1}_{4} } }
\else {\ang { {#1}_0,\ldots,{#1}_{#2}}}\fi\fi}
\newcommand{\xve}[1]{\ifcase 3{#1} {x_1,x_2,x_3} \else 
\ifcase 4{#1} {x_1,x_2,x_3,x_4} \else {x_1,\ldots,x_{#1}}\fi\fi}
\newcommand{\yve}[1]{\ifcase 3{#1} {y_1,y_2,y_3} \else 
\ifcase 4{#1} {y_1,y_2,y_3,y_4} \else {y_1,\ldots,y_{#1}}\fi\fi}
\newcommand{\zve}[1]{\ifcase 3{#1} {z_1,z_2,z_3} \else 
\ifcase 4{#1} {z_1,z_2,z_3,z_4} \else {z_1,\ldots,z_{#1}}\fi\fi}
\newcommand{\ve}[2]{\ifcase 3#2 {{#1}_1,{#1}_2,{#1}_3} \else
\ifcase 4#2 {{#1}_1,{#1}_2,{#1}_3,{#1}_{4}}
\else {{#1}_1,\ldots,{#1}_{#2}}\fi\fi}
\newcommand{\ved}[3]{\ifcase 3#2 {{#1}_{{#3}1},{#1}_{{#3}2},{#1}_{{#3}3}} \else
\ifcase 4#2 {{#1}_{{#3}1},{#1}_{{#3}2},{#1}_{{#3}3},{#1}_{{#3}4}}
\else {{#1}_{{#3}1},\ldots,{#1}_{{#3}{#2}}}\fi\fi}
\newcommand{\fuve}[3]{
\ifcase 3#2
{{#3}({#1}_1),{#3}({#1}_2,{#3}({#1}_3)} \else
\ifcase 4#2
{{#3}({#1}_1),{#3}({#1}_2),{#3}({#1}_3),{#3}({#1}_4)}
\else
{{#3}({#1}_1),\ldots,{#3}({#1}_{#2})}\fi\fi}
\newcommand{\setmathchar}[1]{\ifmmode#1\else$#1$\fi}
\newcommand{\vlist}[2]{%
	\setmathchar{%
		\compound#2\one{#2}\two
		\ifcompound
			({#1}_1,\ldots,{#1}_{#2})
		\else
			\ifcat N#2
				({#1}_1,\ldots,{#1}_{#2})
			\else
				\ifcase#2
					({#1}_0)\or
					({#1}_1)\or
					({#1}_1,{#1}_2)\or 
					({#1}_1,{#1}_2,{#1}_3)\or
					({#1}_1,{#1}_2,{#1}_3,{#1}_4)\else 
					({#1}_1,\ldots,{#1}_{#2})
				\fi
			\fi
		\fi}}

\newif\ifcompound
\def\compound#1\one#2\two{%
	\def\one{#1}
	\def\two{#2}
	\if\one\two
		\compoundfalse
	\else
		\compoundtrue
	\fi}

\newcommand{\xwe}[1]{\ifcase 3{#1} {x_1\wedge x_2\wedge x_3} \else 
\ifcase 4{#1} {x_1\wedge x_2\wedge x_3\wedge x_4} \else {x_1\wedge \cdots \wedge
x_{#1}}\fi\fi}
\newcommand{\we}[2]{\ifcase 3#2 {\ang { {#1}_1\wedge {#1}_2\wedge {#1}_3 } } \else
\ifcase 4{#1} {\ang { {#1}_1\wedge {#1}_2\wedge {#1}_3\wedge {#1}_{4} } }
\else {\ang { {#1}_1\wedge \cdots\wedge {#1}_{#2}}}\fi\fi}
\newcommand{\s}[1]{\s_{#1}}

\newcommand{\monus}{\;\raise.5ex\hbox{{${\buildrel
    \ldotp\over{\hbox to 6pt{\hrulefill}}}$}}\;}

\newcounter{savenumi}

\newtheorem{theoremfoo}{Theorem}[section] 
\newenvironment{theorem}{\pagebreak[1]\begin{theoremfoo}}{\end{theoremfoo}}

\newtheorem{lemmafoo}[theoremfoo]{Lemma}
\newenvironment{lemma}{\pagebreak[1]\begin{lemmafoo}}{\end{lemmafoo}}
\newtheorem{conjecturefoo}[theoremfoo]{Conjecture}

\newtheorem{conventionfoo}[theoremfoo]{Convention}

\newtheorem{porismfoo}[theoremfoo]{Porism}

\newtheorem{gamefoo}[theoremfoo]{Game}

\newtheorem{corollaryfoo}[theoremfoo]{Corollary}

\newtheorem{openfoo}[theoremfoo]{Open Problem}

\newtheorem{exercisefoo}{Exercise}

\newcommand{\fig}[1] 
{
 \begin{figure}
 \begin{center}
 \input{#1}
 \end{center}
 \end{figure}
}

\newtheorem{potanafoo}[theoremfoo]{Potential Analogue}

\newtheorem{notefoo}[theoremfoo]{Note}
\newenvironment{note}{\pagebreak[1]\begin{notefoo}\rm}{\end{notefoo}}

\newtheorem{notabenefoo}[theoremfoo]{Nota Bene}

\newtheorem{nttn}[theoremfoo]{Notation}

\newtheorem{empttn}[theoremfoo]{Empirical Note}

\newtheorem{examfoo}[theoremfoo]{Example}

\newtheorem{dfntn}[theoremfoo]{Definition}

\newtheorem{propositionfoo}[theoremfoo]{Proposition}
\newenvironment{proposition}{\pagebreak[1]\begin{propositionfoo}}{\end{propositionfoo}}

\newenvironment{proof}
    {\pagebreak[1]{\narrower\noindent {\bf Proof:\quad\nopagebreak}}}{\QED}

\newcommand{\yyskip}{\penalty-50\vskip 5pt plus 3pt minus 2pt}
\newcommand{\blackslug}{\hbox{\hskip 1pt
        \vrule width 4pt height 8pt depth 1.5pt\hskip 1pt}}
\newcommand{\QED}{{\penalty10000\parindent 0pt\penalty10000
        \hskip 8 pt\nolinebreak\blackslug\hfill\lower 8.5pt\null}
        \par\yyskip\pagebreak[1]}

\newcommand{\BBB}{{\penalty10000\parindent 0pt\penalty10000
        \hskip 8 pt\nolinebreak\hbox{\ }\hfill\lower 8.5pt\null}
        \par\yyskip\pagebreak[1]}

\newtheorem{factfoo}[theoremfoo]{Fact}

\newenvironment{block}{\begin{list}{\hbox{}}{\leftmargin 1em
    \itemindent -1em \topsep 0pt \itemsep 0pt \partopsep 0pt}}{\end{list}}

\begin{document}

\centerline{\bf $\sum_{p\le n} \frac{1}{p}=\ln(\ln(n)+O(1)$: An Exposition}

\centerline{\bf by William Gasarch and Larry Washington}

\section{Introduction}

It is well known that 
$\sum_{p\le n} \frac{1}{p}=\ln(\ln(n)) + O(1)$
where $p$ goes over the primes.
We give several known proofs of this. 

We first present a proof that
$\sum_{p\le n} \frac{1}{p} \ge \ln(\ln(n)) + O(1)$.
This is based on Euler's proof that $\sum_p \frac{1}{p}$ diverges.
We then present three proofs that 
$\sum_{p\le n}\frac{1}{p}\le \ln(\ln(n)) + O(1)$.
The first one, essentially due to Mertens,  
does not use the prime number theorem.
The second and third one do use the prime number theorem 
and hence are shorter.

For a complete treatment of Merten's proof that $\sum_p \frac{1}{p}$ diverges,
and how it compares with modern treatments, see the scholarly work of
Villarino~\cite{mertens}. 

\section{Euler's Proof that $\sum_{p\le n} \frac{1}{p} \ge \ln(\ln(n)) + O(1)$}

The proof here follows the one in \cite{kraftwash}.

\begin{lemma}\label{le:ln}
For $0\le x\le 1/2$, $-\ln(1-x) \le x+x^2$.
\end{lemma}

\begin{proof}
$-\ln(1-x) = \int_0^x \frac{1}{1-t} dt.$
For $0\le t \le 1/2$, $\frac{1}{1-t}\le 1+2t$. 
Hence
$$
-\ln(1-x)=\int_0^x\frac{1}{1-t} dt\le \int_0^x(1+2t)dt=x+x^2.
$$
\end{proof}

\begin{theorem}
$\sum_{p\le n} \frac{1}{p} \ge \ln(\ln(n)) + O(1)$.
\end{theorem}

\begin{proof}
Clearly 
$$\sum_{j=1}^\infty \frac{1}{j} = 
(1-\frac{1}{2} + \frac{1}{2^2}+\cdots)
(1-\frac{1}{3} + \frac{1}{3^2}+\cdots)\cdots =
\frac{1}{1-2^{-1}}\times
\frac{1}{1-3^{-1}}\times \cdots
$$
which we rewrite as
$$\sum_{j=1}^\infty \frac{1}{j} = \prod_{p} (1-p^{-1})^{-1}$$
We need a finite version of this statement. 
Let $S_n$ be the set of natural numbers whose prime factors $p$ are all $\le n$.
Then
$$
\sum_{j\in S_n} \frac{1}{j} = \prod_{p\le n} (1-p^{-1})^{-1}.
$$
Clearly $\sum_{j\le n} \frac{1}{j}  \le \sum_{j\in S_n} \frac{1}{j}$.
By integration $\ln n \le \sum_{j\le n} \frac{1}{j}$. Hence we have
$$\ln(n) \le \sum_{j\le n} \frac{1}{j} \le \sum_{j\in S_n} \frac{1}{j} = \prod_{p\le n} (1-p^{-1})^{-1}$$
$$\ln(\ln(n)) \le \sum_{p\le n} -\ln(1-p^{-1}).$$
By Lemma~\ref{le:ln} 
$$\sum_{p\le n} -\ln(1-p^{-1}) \le 
  \sum_{p\le n} \frac{1}{p} + \frac{1}{p^2}.$$
Putting this all together we get
$$
\sum_{p\le n}\frac{1}{p}\ge \ln(\ln(n))-\sum_{p\le n}\frac{1}{p^2}
$$
Since the second sum is bounded by $\sum_{i=1}^\infty \frac{1}{i^2}$,
which converges, we have
$$
\sum_{p\le n}\frac{1}{p}\ge \ln(\ln(n))-O(1).
$$
\end{proof}

\begin{note}
If the above proof is done more carefully with attention paid to the constants
you can obtain 
$\sum_{p\le n}\frac{1}{p}\ge \ln(\ln(n))-0.48$. See \cite{kraftwash}.
\end{note}

\section{Mertens Proof that Does Not Use the Prime Number Theorem}

This is adapted from Landau's book \cite{landau}. He works a little harder and gets $o(1)$ instead of $O(1)$.

We first need a weak form of the prime number theorem.

\begin{lemma}
$\pi(x) = O(x/\ln x)$.
\end{lemma}
\begin{proof} Let $n$ be a positive integer. Clearly every prime $p$ with $n<p\le 2n$ occurs in the prime factorization of
the binomial coefficient $\binom{2n}{n}$. Therefore,
$$
n^{\pi(2n)-\pi(n)} = \prod_{n<p\le 2n} n \le \prod_{n<p\le 2n} p \le \binom{2n}{n} \le (1+1)^{2n} = 4^n.
$$
Taking logs yields
$$
\pi(2n)-\pi(n) \le \frac{n\ln 4}{\ln n} \le 2\ln 4 \left(\frac{2n}{\ln(2n)} - \frac{n}{\ln n}\right)
$$
for $n\ge 8$. If $y\ge 16$ is a real number, let $2n$ be the largest even integer with $2n\le y$. 
Then $\pi(y)-\pi(2n)\le 1$ and $|\pi(y/2)-\pi(n)|\le 1$. By increasing $2\ln 4 $ to $4$ we can absorb these errors and obtain
$$
\pi(y)-\pi(y/2)\le 4\left(\frac{y}{\ln y} - \frac{y/2}{\ln(y/2)}\right)
$$
for $y\ge 16$. Adding up this inequality for $y=x, x/2, x/4, \dots$ yields
$$
\pi(x)-\pi(16)\le 4\left(\frac{x}{\ln{x}}\right).
$$
This yields the lemma.
\end{proof}

We now need a result that is interesting in its own right.
\begin{proposition}
$$
\sum_{p\le x} \frac{\ln p}p = \ln x + O(1).
$$
\end{proposition}

\begin{proof}
If $n$ is a positive integer and $p$ is a prime, the power of $p$ dividing $n!$ is
$$
\left\lfloor\frac{n}{p}\right\rfloor + \left\lfloor\frac{n}{p^2}\right\rfloor + \left\lfloor\frac{n}{p^3}\right\rfloor + \cdots.
$$
Therefore,
$$
\ln(n!)=\sum_{p\le n} \ln p \left(\left\lfloor\frac{n}{p}\right\rfloor + \left\lfloor\frac{n}{p^2}\right\rfloor + \left\lfloor\frac{n}{p^3}\right\rfloor + \cdots\right).
$$
Changing $\left\lfloor n/p\right\rfloor$ to $n/p$ introduces an error of most 1, so we have
$$
\sum_{p\le n} \ln p \left\lfloor\frac{n}{p}\right\rfloor = n\sum_{p\le n} \frac{\ln p}{p} + O(\sum_{p\le n} \ln p).
$$
Since there are $\pi(n)$ terms in the sum, Lemma 1 implies  that 
$$O(\sum_{p\le n} \ln p)= O(\pi(n)\ln n) = O(n).$$

Let's treat the higher terms:
$$
\left\lfloor\frac{n}{p^2}\right\rfloor + 
\left\lfloor\frac{n}{p^3}\right\rfloor + 
\cdots < 
\frac{n}{p^2}\left(1+p^{-1}+p^{-2}+\cdots\right) = \frac{n}{p^2-p}.
$$
Therefore,
$$
\sum_{p\le n} \ln p 
\left(\left\lfloor\frac{n}{p^2}\right\rfloor + 
\left\lfloor\frac{n}{p^3}\right\rfloor + 
\cdots\right)\le n\sum_{p\le n} \frac{\ln p}{p^2-p}=O(n)
$$
since $\sum \ln p/(p^2 -p) \le \sum_{j\ge 2} \ln j/(j^2-j)$, which converges.

Stirling's formula says that
$$
\ln(n!) = n\ln n+O(n)
$$
(this weak form can be proved by comparing $\sum \ln j$ with $\int \ln t\, dt$).
Putting everything together yields
$$
n\ln n + O(n) = n\sum_{p\le n} \frac{\ln p}{p} + O(n).
$$
Dividing by $n$ yields the proposition for $x=n$. 
The error introduced by changing from $x$ to $n=\left\lfloor{x}\right\rfloor$ is absorbed
by $O(x)$, so the proposition is proved. 
\end{proof}

The following lemma is well known. It is an analog of integration by
parts for summations. It is easily proven by induction on $n$.

\begin{lemma}\label{le:sumparts}
Let both $f_1,f_2,\ldots$ and $g_1,g_2,\ldots$ be sequences 
of complex numbers.
Then, for all $m\le n$,
$$
\sum_{i=m}^n f_i(g_{i+1}-g_i)=
f_{n+1}g_{n+1}-f_mg_m-\sum_{i=m}^n g_{i+1}(f_{i+1}-f_i).
$$
\end{lemma}

We can now prove the theorem. 

\begin{theorem} 
$\sum_{p\le x} \frac 1p = \ln\ln x + O(1).$
\end{theorem}

\begin{proof}
We have
$$
f(x)=\sum_{p\le x} \frac{\ln p}{p} = \ln x + r(x),
$$
where $r(x)= O(1)$.
Then
\begin{align*}
\sum_{p\le x} \frac 1p &= \sum_{p\le x} \frac{\ln p}{p}\frac{1}{\ln p} = \sum_{n=2}^{x} \frac{f(n)-f(n-1)}{\ln n}\\
&= \sum_{n=2}^x \frac{\ln n - \ln(n-1)}{\ln n} + \sum_{n=2}^x \frac{r(n)-r(n-1)}{\log n}.
\end{align*}
Since 
$$
\ln n - \ln(n-1) = -\ln\left(1-\frac 1n\right) = \frac 1n + O(1/n^2),
$$
and
$$
\sum_{n=2}^x \frac{1}{n\ln n} = \ln \ln x + O(1),
$$
we find that
$$
\sum_{n=2}^x \frac{\ln n - \ln(n-1)}{\ln n} = \ln\ln x + O(1).
$$
Summation by parts yields
\begin{align*}
\sum_{n=2}^x \frac{r(n)-r(n-1)}{\log n}&= \sum_{n=2}^x r(n)\left(\frac{1}{\ln n} - \frac{1}{\ln(n+1)}\right) + \frac{r(\left\lfloor x\right\rfloor)}{\ln(\left\lfloor x\right\rfloor +1)}\\
&=O\left(\sum_{n=2}^x \frac{1/n}{(\ln n)^2}\right) + O(1) = O(1).
\end{align*}
Putting everything together yields
the theorem. 
\end{proof}

\section{A Proof that uses Summation by Parts}

In this section we give the standard way to estimate $\sum 1/p$ using the Prime Number Theorem.

%
%
%
%

\begin{theorem}
$\sum_{p\le n} \frac{1}{p} = \ln(\ln(n)) + O(1)$.
\end{theorem}

\begin{proof}
Let $\pi(i)$ be the number of primes $\le i$.
Let $g(i)=\pi(i-1)$ and $f(i)=\frac{1}{i}$.
Let $m=2$.
Plugging these into Lemma~\ref{le:sumparts} yields
$$
\sum_{i=2}^n \frac{1}{i}(\pi(i) -\pi(i-1))=
\frac{1}{n+1}\pi(n) - \frac{1}{2}\pi(1) - 
\sum_{i=2}^n \pi(i)(\frac{1}{i+1}-\frac{1}{i}).
$$
We need:
\begin{itemize}
\item
$\pi(i)-\pi(i-1)$ is 1 if $i$ is prime but 0 otherwise.
\item
$\pi(n) = \frac{n}{\ln n}+ O(\frac{n}{\ln^2 n})$ by the Prime Number Theorem (when it is proved with an error term).
\end{itemize}
We have
$$
\pi(i)(\frac{1}{i+1}-\frac{1}{i})= \frac{\pi(i)}{i(i+1)}=\frac{1}{(i+1)\ln i} + O\left(\frac{1}{(i+1)\ln^2 i}\right)
$$
by the Prime Number Theorem.
But this equals
$$
\frac{1}{i\ln i} -\frac{1}{i(i+1)\ln i}+O\left(\frac{1}{(i+1)\ln^2 i}\right)= \frac{1}{i\ln i}+ O\left(\frac{1}{(i+1)\ln^2 i}\right).
$$
Therefore,
$$
\sum_{p\le n} \frac{1}{p} = \sum_{i=2}^n \frac{1}{i\ln i}+ O\left(\frac{1}{(i+1)\ln^2 i}\right) = \ln(\ln(n)) + O(1),
$$
where we have used
$$
\sum_{i=2}^n \frac{1}{i\ln i}=\int_2^n\frac{1}{x\ln x} dx +O(1) = \ln(\ln(x)) + O(1)
$$
and
$$
\sum_{i=2}^n \frac{1}{(i+1)\ln^2 i} = O(1)
$$
by the Integral Test.
\end{proof}

\section{A Proof that uses Integration by Parts}

This is the same as the previous proof, with the summation by parts replaced by integration
by parts in a Stieltjes integral.

\begin{theorem}
$\sum_{p\le n} \frac{1}{p} = \ln(\ln(n)) + O(1)$.
\end{theorem}

\begin{proof}
The preceding proof can be rewritten using Stieltjes integrals:
$$\sum_{p\le x} \frac{1}{p} = \int_{1.9}^x \frac{1}{t} d\pi(t).$$
Integration by parts yields
$$\frac{\pi(x)}{x} + \int_{1.9}^x \frac{\pi(t)}{t^2} dt.$$
We use the Prime Number Theorem approximation $\pi(x) = \frac{x}{\ln x} + O(\frac{x}{\ln^2 x})$ to obtain
$$
\frac{1}{\ln x}+\int_{1.9}^x\frac{1}{t\ln t}+O(\int_{1.9}^x\frac{1}{t\ln^2 t})=\ln(\ln(x))+O(1).
$$
\end{proof}

\section{What Else is Known}

Rosser and Schoenfeld~\cite{primesrecip} have shown that, when $n\ge 286$,
$$
\ln(\ln n)-\frac{1}{2(\ln n)^2}+B\le\sum_{p\le n}\frac{1}{p}\le \ln(\ln n) + \frac{1}{(2\ln n)^2} + B,
$$
\noindent
where $B=0.261497212847643$.

Even though the sum $\sum_{p\le n} \frac{1}{p}$ diverges, it
grows very slowly:
\begin{itemize}
\item
$\sum_{p\le 10} \frac{1}{p} = 1.176$
\item
$\sum_{p\le 10^6} \frac{1}{p} = 2.887$
\item
$\sum_{p\le 10^9} \frac{1}{p} = 3.293$
\item
$\sum_{p\le 10^{100}} \frac{1}{p} \sim 5.7$
\end{itemize}


\begin{thebibliography}{1}

\bibitem{kraftwash}
J.~Kraft and L.~Washington.
\newblock {\em An introduction to {N}umber {T}theory and {C}ryptogaphy}.
\newblock CRC Press, 2014.

\bibitem{landau}
E.~Landau.
\newblock {\em Handbuch der {L}ehre von der {W}erteilung der {P}rimzahlen}.
\newblock Chelsea Publhsing Co, 1953.

\bibitem{primesrecip}
J.~Roser and L.~Schoenfeld.
\newblock Approximate formulas for some properties of prime numbers.
\newblock {\em Illinois Journal of Mathematics}, pages 64--94, 1962.

\bibitem{mertens}
M.~B. Villarino.
\newblock Merten's proof of {M}ertens' theorem, 2005.
\newblock \url{http://arxiv.org/abs/1205.3813}.

\end{thebibliography}

\end{document}